%% file: article.tex
\documentclass[1p]{elsarticle}

\usepackage{mathrsfs}
\usepackage{amssymb,amsthm}
\usepackage{amsmath,amsfonts}
\usepackage{mathtools}
\usepackage{caption}
\usepackage{xcolor}
\usepackage{comment}
\usepackage{lipsum}
\usepackage{graphicx}
\usepackage{epstopdf}
\usepackage{algorithmic}
\usepackage{bbold}
\usepackage{tikz}
\usepackage{subfig}
\usepackage{adjustbox}
\usepackage[noresetcount,vlined,ruled]{algorithm2e} 
\usepackage[compact]{titlesec}
\usepackage{autobreak} 
\usepackage{hyperref}
\allowdisplaybreaks

\ifpdf
  \DeclareGraphicsExtensions{.eps,.pdf,.png,.jpg}
\else
  \DeclareGraphicsExtensions{.eps}
\fi

\theoremstyle{remark}

\newtheorem{proposition}{Proposition}

\journal{SIAM Journal of Control and Optimization}

\begin{document}
\begin{frontmatter}

\title{Solving Bang-Bang Problems Using The Immersed Interface Method and Integer Programming}

\author[ncsu,anl,llnl]{Ryan H. Vogt\corref{cor1}}
\ead[email]{rvogt2@ncsu.edu}

\author[ncsu]{Sarah Strikwerda}
\ead[email]{slstrikw@ncsu.edu}

\cortext[cor1]{Corresponding author}
\address[ncsu]{Department of Mathematics, North Carolina State University, Raleigh, NC 27695, USA.}  
\address[anl]{Mathematics and Computer Science Division, Argonne National Laboratory, Lemont, IL 60439, USA.}  
\address[llnl]{Center for Applied Scientific Computing, Lawrence Livermore National Laboratory, Livermore, CA 94550, USA.}

\renewcommand{\thefootnote}{\arabic{footnote}}

\begin{abstract}
 In this paper we study numerically solving optimal control problems with bang-bang control functions. We present a formal Lagrangian approach for solving the optimal control problem, and address difficulties encountered when numerically solving the state and adjoint equations by using the immersed interface method. We note that our numerical approach does not approximate the discontinuous control function with smooth functions, instead we solve the true bang-bang optimal control problem. Our approach for solving the optimal control problem uses an adjoint-based gradient. We use the gradient in our first-order trust-region method to generate a local minimizing control. We present detailed numerical results to demonstrate the effectiveness of our method.
\end{abstract}


\end{frontmatter}

\input{paper}

\end{document}

%% file: paper.tex
\section{Introduction} 
Optimal control problems are frequently studied due to their importance in a wide range of fields including economics, biological systems, and engineering\cite{chen2005optimal,1273042}. In these problems, we seek to find a control that minimizes a cost functional subject to a differential equation. These problems can quickly become very complex. Therefore, in many applications a solution that only takes on a few discrete values is sought after. Such a control is called a bang-bang control. The most common bang-bang control is one which takes on the value of $0$ or $1$. This is called a bang-bang (on-off) control, and it abruptly changes between extreme values. The change between these values is referred to as a "bang".
Bang-bang controls have been used to study how to control the population of invasive species \cite{srinivasu2011role}, regulate chemotherapy \cite{wang2016optimal}, and make decisions about stock options \cite{chen2005optimal}.

In the current state of the art approaches to solving bang-bang optimal control problems \cite{bestehorn2019switching,antal2017proactive,de2019mixed}, which are a special case of mixed-integer optimal control problems (MIOCPs), several assumptions are made. First, the solution to the different equation constraint is continuous. Second, that when numerically solving the differential equation constraint, the switching points must align with the computational mesh. In this work, we provide a framework to solve bang-bang optimal control problems without assuming regularity in the state, or the computational mesh having to align with the switching points. In this paper we  develop finite-difference schemes associated with the bang-bang control problem that are consistent, stable and convergent; using the immersed interface method (IIM). Next we introduce an adjoint-based trust-region method to solve the reduced-space formulation of the bang-bang control problem.

We take a formal Lagrangian approach, assuming a reduced-space model. In other words, we assume that given a control $v$, there exists a unique state T dependent on v. This is to say that the state-to-control operator $T(v)$ is injective. We note that our method can also readily be applied to full-space formulations as well. In our case of a  reduced-space formulation, we derive a gradient for $v$ by solving the strong form of the state and adjoint equation. We address the lack of regularity of the state and adjoint equation, with IIM. We then use this adjoint-based gradient in a trust-region method for binary optimization to find a local minimizing bang-bang control.

To demonstrate our approach for solving bang-bang problems, we focus on numerically solving the following temperature control problem:

\begin{equation}\label{eq:genproblem}
\begin{aligned}
& \underset{T,\; v}{\text{minimize}} & & \frac{1}{2} \int_{\Omega} (T - \hat{T})^2dt\\
& \text{subject to} 
  & &  \frac{dT}{dt} = -K(T-T_s) +Cw +f \quad \text{ on} \; \Omega,\\
   & & &   \Omega = \bigcup_{i=1}^{N }\Omega_i, \\
   & & &   w(t) = \sum_{i=1}^{N} v_i \mathbb{1}_{\Omega_i \setminus \partial \Omega_i  }(t)  \\,
    & & &  v_i \in \{0,1\} \quad \forall i=1...N,   \\
    & & & \Omega_i =[\tau_i,\tau_{i+1}]  \quad \forall \; i=1,...,N, \\
    & & & \tau_1 = 0, \quad \tau_{N+1} = t_{\text{final}}, \\
   & & & \Omega = [0,t_{\text{final}}],   \\
   & & & T(0) = T_0
 \end{aligned}
 \end{equation}

where we assume that $T(t)$ is the temperature of a body at time $t$, $\hat{T}(t)$ is the target temperature profile, $f \in C(\Omega_i \setminus \partial \Omega_i)$, and $w(t)$ the control function is either "on" or "off" at a time $t$. We represent the control function $w(t)$ as a linear combination of indicator functions $\mathbb{1}$, and assign the control function value to be $v_i$ on $\Omega_i \setminus \partial \Omega_i$. In this paper we assume that the support of the bang-bang choices, defined by $\Omega_i$, are equal length i.e,
\begin{equation*}
    |\Omega_i| = \frac{T}{N},\quad \forall i=1...N.
\end{equation*} 
We also assume that $K, \; C, \; t_{\text{final}} >0$. 
We note that this assumption does not have to hold for our method to work; it can very well be the case the the intervals lengths can vary, but we choose them to be constant length in this work. We lastly note that we will refer to the objective in this paper as
\begin{equation}
    \mathcal{J} =  \frac{1}{2} \int_{\Omega} (T - \hat{T})^2dt.
\end{equation}
 
Even though this paper focuses on the specific model problem \eqref{eq:genproblem}, we explain at each step taken to solve \eqref{eq:genproblem} how our approach can be tailored to other bang-bang problems.

In this paper we develop a first order, stable convergent method to solve the state and adjoint equations. The method takes advantage of IIM \cite{leveque1994immersed}. The IIM was originally developed to solve elliptic partial-differential equations with discontinuous coefficients, also known as interface problems, with a finite-difference approach. The IIM has been successfully applied to generate high order methods for solving equations where discontinuities are present \cite{linnick2005high,lee2003immersed,li2001immersed}. In our work we use the IIM to solve the adjoint and state equation due to the bang-bang control. We establish a modifed Euler's method that can solve the state and adjoint equation, while also obtaining second order accuracy, stability, and convergence of the numerical solution.

In practice, analytic solutions to optimal control problems are difficult to obtain. In absence of analytic techniques, numerical methods are a powerful tool to solve optimal control problems. In our example here, the bang-bang control makes it difficult to apply existing numerical techniques to solve ODE's because $T \not \in C^1(\Omega)$. One at first glance may look at \eqref{strongstate} and believe that a simple method, such as Euler's method, may be readily applied to solve the equation numerically. However, Eulers method assumes that $T \in C^2(\Omega)$ for the method to be applicable. This motivates our approach of solving the state equation with IIM.

 We note in our framework the effort of numerically solving the bang-bang problem lies in solving the differential equations, and building the gradient. The amount of bangs present in the bang-bang problem, from the perspective of applying our trust-region method, has little impact on the overall computational effort because we solve a sequence of binary knapsack problems which can be solved easily.


\section{The Lagrangian and the relaxed gradient}\label{s:laggrad}
To calculate the gradient of the objective function with the respect to the control parameters $v$, $\frac{\partial \mathcal{J}}{\partial v}$, we use a formal Lagrangian approach. We first relax the condition of $v\in \{ 0,1\}^N$ to a continuous relaxation: $0 \leq v \leq 1$. We do this so that we have the notion of a continuous gradient. However, because this gradient expression is valid for any $0\leq v \leq 1$, then it will be valid when evaluated at a $v\in \{0,1\}^N$. We introduce the Lagrangian:
\begin{align}
   \mathcal{L}(T,\lambda,v) = \frac{1}{2} \int_{\Omega} (T-\hat{T})^2 dt + \int_{\Omega}\lambda \bigg(\frac{dT}{dt} + K(T-T_s) - Cw -f\bigg)dt,   
\end{align}
where $\lambda$ is the adjoint variable  (Lagrangian multiplier). 

Next we derive the weak state and weak adjoint equations by calculating the following variational (Gateaux) derivatives:
\begin{align*}
    &\mathcal{L}_\lambda[\Tilde{\lambda}] = \frac{d}{d\epsilon} \bigg( \mathcal{L}(T,\lambda+ \epsilon \Tilde{\lambda},v) \bigg) \bigg|_{\epsilon=0}=0 \quad (\text{Weak State Equation})\\
    &\mathcal{L}_T[\Tilde{T}] = \frac{d}{d\epsilon} \bigg( \mathcal{L}(T+ \epsilon \Tilde{T},\lambda,v) \bigg) \bigg|_{\epsilon=0}=0 \quad (\text{Weak Adjoint Equation}).
    \end{align*}
    The explicit form of these equations are:
    \begin{align}
        &\int_\Omega  \bigg(\frac{dT}{dt}  +K(T-T_s) -Cw -f\bigg) \Tilde{\lambda}dt=0 \quad \forall \Tilde{\lambda} \in H^1(\Omega), \\
       &\int_{\Omega} \bigg(-\frac{d\lambda}{dt}  +K\lambda + T-\hat{T}\bigg) \Tilde{T}dt = 0\quad \forall \Tilde{T} \in H^1(\Omega)
     \end{align}
The state and adjoint solutions are not $C^1(\Omega)$, however from standard ODE theory we see the state and adjoint equations are $C^1(\Omega_i\setminus \partial \Omega_i)$. Therefore, we write the strong form of these equations restricted to the set  $\Omega_i\setminus \partial \Omega_i \; \forall i$:
 \begin{equation} \begin{cases}
     &\frac{dT}{dt} = -K(T-T_s) +Cw +f \quad (\text{Restricted Strong State Equation}) \label{strongstate}
     \\ 
     &T(0)=T_0
\end{cases}
\end{equation}
\begin{equation}
\begin{cases}
     &-\frac{d\lambda}{dt} = -K\lambda - (T-\hat{T}) + g \quad (\text{Restricted Strong Adjoint Equation})\label{strongadjoint}\\ 
     &\lambda(T) = 0.
 \end{cases}
 \end{equation}
 In the remainder of the paper we refer the restricted state and adjoint equations as just the state and adjoint equations. We note that in the adjoint equation $g=0$. However, we include $g \in C(\Omega_i \setminus \partial \Omega_i)$ to test our adjoint solver in Section \ref{numerics} using the method of Manufactured Solution and the IIM.
 
 The $i$-th component of the gradient with respect to $v_i$ is then:
 \begin{equation}
      \frac{\partial \mathcal{J}}{ \partial v_i} =  \frac{\partial \mathcal{L}}{ \partial v_i} = -\int_{\Omega_i} C \lambda dt.
 \end{equation}
 So to obtain a gradient associated with a fixed $v$, we first solve the strong state equation, then solve the strong adjoint equation. Once we have the adjoint variable, we then integrate a scaled version it over $\Omega_i$ to obtain the $i$-th component of the gradient.

\section{The discretization of the restricted state Equation}\label{s:statediscreization}

In our optimal control problem, and other bang-bang problems, the IIM is readily applicable to solve the state equation. This is because the presence of the bang-bang control can cause, as seen in \eqref{strongstate}, the state solution to not be $C^1(\Omega)$. However, we can readily apply the IIM to this problem despite the lack of regularity of the state, or its derivatives, and have theoretical guarantees that the numerical solution is correct.

For many numerical methods, such as Euler or Runge-Kutta, there is an assumption of high regularity in state. Indeed in these methods, the error terms will contain a high order derivative term. If the solution to an ODE doesn't have the sufficient regularity to apply these numerical schemes, then one cannot guarantee the consistency of the method. If the consistency of the numerical method cannot be established then there is no hope to establish the convergence of the numerical solution to the true solution. This means one cannot be confident that the numerical method approximates the state solution correctly. 

In order to apply IIM to the state equation \eqref{strongstate}, we begin by deriving the jump conditions associated with \eqref{strongstate}. The jump conditions describe how the solution and it's derivatives jump when a bang occurs. The jump conditions are used to create our numerical scheme. In our problem, we introduce $N$ binary variables to model the bang-bang control, so there would be at most $N-1$ bangs (switches). Each one of these bangs will occur at
\begin{align}
    t = \alpha_i = \Omega_i \cap \Omega_{i+1} \quad i=1,...,N-1.
\end{align}
 For our state scheme, we focus on developing a first-order method, so we need only to derive the the jump in the state and its first derivative over these bang points. We note if one wishes to apply the IIM to yield a method of order $m$, then one must derive the jump conditions up to the $m$-th derivative.

At each bang time $t=\alpha_i$, we define the jump in the state and its first derivative to be:
\begin{align}\label{jumpconditionsfull}
&[T]_{t=\alpha_i} = T^+(\alpha_i) - T^-(\alpha_i) \\
&\bigg[\frac{dT}{dt}\bigg]_{t=\alpha_i} = \frac{dT}{dt}^+(\alpha_i) - \frac{dT}{dt}^-(\alpha_i),
\end{align}
where we we denote "$+$" to represent the solution and its derivative to the right of the bang point and "$-$" to the left of the bang point. In general we write the $k$-th order derivative as:
\begin{align*}
    \bigg[\frac{d^kT}{dt^k}\bigg]_{t=\alpha_i} = \frac{d^kT}{dt^k}^+(\alpha_i) - \frac{d^kT}{dt^k}^-(\alpha_i),
\end{align*}
where we assume that the zero order derivative is just the state itself.

In many physical examples, even though the derivative is discontinuous (therefore a non-zero jump) the state is continuous so
\begin{align*}
    [T]_{t=\alpha_i} =0.
\end{align*}
This would be known as a direct IIM because we know the jump in the state explicitly. However, in this work we make no such assumption since for general problems there is no manner to know that the state is continuous across the bang points. Therefore in this work we assume it to be unknown i.e
\begin{align}\label{zeroaug}
    [T]_{t=\alpha_i} = q_i,
\end{align}
where $q_i$ is to be determined. When the jump conditions are not known explicitly, the method is called the Augmented IIM \cite{li2001immersed}, and $q_i$ for $i=1,..,N-1$ are called the augmented variables. From \eqref{strongstate} we see that the jump condition for the derivative of the state is
\begin{align}
    \bigg[\frac{dT}{dt}\bigg]_{t=\alpha_i} = -K[T]_{t=\alpha_i} +C[w]_{t=\alpha_i} +[f]_{t=\alpha_i},
\end{align}
because we assume $K,$ and $T_s$ to be constant in this work. We then express the jump of the first-derivative in terms of the augmented variables \eqref{zeroaug}:
\begin{align*}\label{firstaug}
     \bigg[\frac{dT}{dt}\bigg]_{t=\alpha_i} = -Kq_i +C[w]_{t=\alpha_i} +[f]_{t=\alpha_i}.
\end{align*}

Next we focus on deriving the IIM scheme for our problem. We discretize the interval $[0,t_{\text{final}}]$ into $N_t$ steps with $t_n =n\Delta t$, $n=0,...N_t$ and $\Delta t = \frac{t_{\text{final}}}{N_t}$. The idea of the IIM scheme is to use a standard scheme away from the bang points and derive a "corrected" scheme around the bang points to obtain a consistent scheme. In this work we base our scheme around Euler's method. We say that if $t_n \leq \alpha_i \leq t_{n+1}$ for some $i$ then $t_n$ is an irregular grid point. If this is not the case, then the grid point is said to be regular. At regular grid points we apply Euler's method:
\begin{align}
    T^{n+1} = T^n  + \Delta t(-K(T^n-T_s) +Cw^n +f^n) + \mathcal{O}(\Delta t).
\end{align}
In the case $t_n$ is an irregular grid point we use the method of unknown coefficients:
\begin{align}\label{statemethod}
    c_1^nT^{n+1} + c_2^n T^{n} - \bar{C}^n(c_1,c_2) = -K(T^n-T_s) +Cw^n +f^n + \mathcal{O}(\Delta t),
\end{align}
where the goal is to find $c_1^n$, $c_2^n$ and the correction term, $\bar{C}^n(c_1^n,c_2^n)$, that maintains the same order of error as the standard scheme at the irregular grid points, which in this case is first-order accuracy. Since a grid point in out interval is either regular or irregular, we will obtain the order of accuracy of the scheme globally even though the state nor its derivatives need to be continuous on the interval.

We use the Taylor series to determine $c_1,c_2$ and the correction term with first order accuracy. We expand the derivative in terms of the "$-$" side about $\alpha_i$:
\begin{align} \label{stateapprox}
    &\frac{dT}{dt}^-(\alpha_i) = c_1T^{n+1}+c_2T^n = \\ \nonumber
    & c_1^n\bigg(T^+(\alpha_i) + \frac{dT}{dt}^+(\alpha_i)(t_{n+1}-\alpha_i) + \mathcal{O}\bigg((t_{n+1}-\alpha_i)^2 \bigg)  \bigg) + \\ \nonumber
    & c_2^n\bigg(T^-(\alpha_i) + \frac{dT}{dt}^-(\alpha_i)(t_{n}-\alpha_i) + \mathcal{O}\bigg((t_{n}-\alpha_i)^2\bigg)  \bigg).
\end{align}
Note that we expand $t_{n+1}$ in terms of "$+$" and $t_n$ in terms of "$-$" because they are to the right and left of the bang point $\alpha_i$, respectively. 

Next we use the jump conditions \eqref{jumpconditionsfull} to write $T^+$ and $\frac{dT}{dt}^+$ in terms of $T^-$ and $\frac{dT}{dt}^-$ so that we can equate the two sides to find $c_1,c_2$ and the correction term. When equating the two sides, we find that $c_1,c_2$ satisfy the following system of equation:
\begin{align*}
    &c_1^n + c_2^n = 0\\
    &c_1^n(t_{n+1}-\alpha_i) + c_2^n(t_n-\alpha_i) =1,
\end{align*}
and the analytic solution to this system is
\begin{align}\label{stateconstants}
    c_1^n = \frac{1}{\Delta t}, \quad c_2^n=-\frac{1}{\Delta t}.
\end{align}
The corresponding correction term is then
\begin{align}\label{statcorrection}
    \bar{C}^n = \frac{1}{\Delta t}\bigg( q_i + ( -Kq_i +c[w]_{t=\alpha_i} + [f]_{t=\alpha_i})(t_n-\alpha_i)\bigg).
\end{align}

\section{The discretization of the restricted adjoint equation}\label{s:adjdiscreization}
We build on Section \ref{s:statediscreization} to solve the adjoint equation numerically. Like the state equation, we first derive the jump conditions. Due to the fact that the adjoint equation is solved backwards in time, we introduce the notion of "backward" jump in terms of our previous definitions $[ \cdot]_{t=\alpha_i}^{\leftarrow} = -[\cdot]_{t=\alpha_i}$ i.e
\begin{align}\label{adjjumps}
    &[\lambda]_{t=\alpha_i}^{\leftarrow} = - [\lambda]_{t=\alpha_i} = \lambda^- - \lambda^+=q_i^\lambda ,\\
     &\bigg[\frac{d\lambda}{dt}\bigg]_{t=\alpha_i}^{\leftarrow} = - \bigg[\frac{d\lambda}{dt}\bigg]_{t=\alpha_i} = Kq_i^\lambda - q_i - [g]_{t=\alpha_i}^{\leftarrow},
\end{align}
and we introduce a new set of augmented variables $q_i^{\lambda}$ that are to be determined. We maintain the same time mesh as the state equation, and move backward instead of forward in time. When moving backwards, from $t_n$ to $t_{n-1}$, we say $t_n$ is an irregular grid point for the adjoint discretization if $t_{n-1} \leq \alpha_i \leq t_n$, and we say $t_n$ is a regular grid point otherwise. In the case $t_n$ is a regular grid point, we use a backward finite-difference approximation of the adjoint variable's first derivative:
\begin{align*}
 \frac{\lambda^n - \lambda^{n-1}}{\Delta t}= -K\lambda^n - (T^n-\hat{T}^n) + g^n + O(\Delta t).
\end{align*}
For an irregular grid point, we expand on the $'+'$ side, because we are marching backwards:
\begin{align*}
    \frac{d \lambda}{dt}^+(\alpha_i) = c_1^{n,\lambda}( \lambda^+(\alpha_i) + \frac{d \lambda}{dt}^+(\alpha_i)(t_n -\alpha_i) ) + c_2^{n,\lambda}( \lambda^-(\alpha_i) + \frac{d \lambda}{dt}^-(\alpha_i)(t_{n-1} -\alpha_i)).
\end{align*}
Like the state equation, we use the method of manufactured coefficients, and the jump conditions for the adjoint variables \eqref{adjjumps} to find the system of equations that $c_1^{n,\lambda}, c_2^{n,\lambda}$ satisfy which we find to be
\begin{align*}
    &c_1^{n,\lambda} + c_2^{n,\lambda}=0 \\
    &c_1^{n,\lambda}(t_n-\alpha_i) + c_2^{n,\lambda}(t_{n-1}-\alpha_i)=1. 
\end{align*}
Solving the systems of equations, we find that
\begin{align*}
     c_1^{n,\lambda} = -\frac{1}{\Delta t}, \quad  c_2^{n,\lambda} = \frac{1}{\Delta t},
\end{align*}
with corresponding correction term 
\begin{align*}
    \bar{C}^{n,\lambda} = \frac{1}{\Delta t} \bigg(q_i^{\lambda} + (Kq_i^{\lambda} -q_i -[g]_{t=\alpha_i}^{\leftarrow})(t_{n-1}-\alpha_i)\bigg).
\end{align*}
The scheme at irregular grid points is then
\begin{align*}
    \frac{\lambda^n -\lambda^{n-1}}{\Delta t} + \bar{C}^{n,\lambda}  = -K\lambda^n - (T^n-\hat{T}^n) + g^n + O(\Delta t)
\end{align*}

\section{Stability, consistency, and convergence of the state and adjoint finite-difference schemes}

We must discuss stability, consistence, and convergence of numerical schemes, for regular and irregular grid points independently. Because a grid point is either regular or irregular, by showing stability, consistency, and convergence in each case we show the overall scheme is stable, consistent, and convergent. 

In practice, the scheme used for the IIM at regular grid points is a standard method, so the stability, consistency, and convergence of the method is already established. What is left to do is to establish the stability, consistency, and convergence for the scheme at irregular grid points.

In practice, we note that consistency of a method is more readily obtainable than stability. We begin by showing that the numerical scheme is consistent for the state scheme. \begin{proposition}
The state scheme is consistent with local truncation error of $\mathcal{O}(\Delta t)$ at both regular and irregular grid points, therefore the method has global truncation error $\mathcal{O}(\Delta t)$ .
\end{proposition}
\begin{proof}
We begin by showing the scheme is consistent at irregular grid points. We substitute the finite difference weights found in \eqref{stateconstants} and correction term, \eqref{statcorrection}, into the numerical scheme for the state equation in \eqref{statemethod}. We note that the local truncation error of the method $\tau$ becomes
\begin{align*}
    \tau = \frac{1}{\Delta t} \mathcal{O}((t_{n+1} - \alpha_i)^2) - \frac{1}{\Delta t} \mathcal{O}((t_{n} - \alpha_i)^2).
\end{align*}
We note that at the irregular grid point $t_n$, $t_n \leq \alpha_i \leq t_{n+1}$, and $t_{n+1} -t_n = \Delta t$. Therefore,
\begin{align*}
    &\tau = \frac{1}{\Delta t} \mathcal{O}((t_{n+1} - \alpha_i)^2) - \frac{1}{\Delta t} \mathcal{O}((t_{n} - \alpha_i)^2) \leq \frac{1}{\Delta t} \mathcal{O}(\Delta t^2) - \frac{1}{\Delta t} \mathcal{O}(\Delta t^2) = \mathcal{O}(\Delta t)\\
    &\implies \tau = \mathcal{O}(\Delta t),
\end{align*}
which is to say that the method is a first-order method at the irregular grid points.
We then observe that
\begin{align*}
   \lim_{\Delta t \to 0} \tau = 0,
\end{align*}
so the method at irregular grid points is consistent as well.

Next we refer to \cite{gautschi1997numerical} to show that the scheme is consistent at regular grid points. We note that the scheme at regular grid points is the standard forward Euler method. This scheme has a local truncation error of $\mathcal{O}(\Delta t)$, so the scheme is first-order at regular grid points, and moreover is consistent because  
\begin{align*}
   \lim_{\Delta t \to 0} \mathcal{O}(\Delta t) = 0.
\end{align*}
We have shown that the scheme for the state equation is first-order accurate and consistent at both the regular and irregular grid points, so the state scheme is consistent and first-order accurate.
\end{proof}
We omit a proof that shows the of the adjoint scheme is first-order and consistent because the argument to prove the result is identical to the state equation proof, just applied to the adjoint scheme. 

Next, we show the stability of the state and adjoint schemes by referring to the literature. In both cases we write the equations in the form:
\begin{align*}
    &\frac{dT}{dt} = \Psi_1(T,t), \\
    &\frac{d \lambda}{dt} = \Psi_2(\lambda,T,t).\\
\end{align*}
Because $\Psi_1$ and $\Psi_2$ are Lipschitz with respect to the variables $T$, and $\lambda$ on $\Omega_i \setminus \partial \Omega_i$ for all $i=1..N$ , then a one step method is stable as shown in \cite[Theorem 5.7.1, p. 345]{gautschi1997numerical}. Our methods for the state and adjoint scheme are one step methods, and the correction terms at irregular grid points for both the state and adjoint scheme are bounded constants (which means $\Psi_1$ and $\Psi_2$ are Lipschitz for regular and irregular grid points) so then the overall scheme is stable.  

Now we show the convergence of the state and adjoint schemes. We again refer to the literature, where if a one step method is consistent and stable, then it is convergent \cite[Theorem 5.7.2, p. 348]{gautschi1997numerical}. Since we have proven the consistency and stability of the method for both regular and irregular grid points we can establish the convergence of the solution at regular and irregular grid points due to this theorem. Therefore, we have established the consistency, stability, and convergence of both the state and adjoint numerical schemes.

We note that in our framework even though the solutions are not continuous on the entire mesh, it is on the pieces, which allows us to establish the consistency, stability and convergence of our method due to our IIM approach. If a schemes consistency, stability, and convergence has been established for differential equations with high regularity, those proofs can traditionally be altered slightly to establish these properties for the irregular grid points. To our knowledge this has not been possible for finite-difference methods applied to bang-bang control problems before this work. 
\section{The Discrete State and Adjoint Systems}\label{S:disstateobj}
Suppose we are given a control $v$, then augmented IIM for the state equation can be written as the block matrix system:
\begin{equation}
    \begin{bmatrix}
\boldsymbol{A_s} & \boldsymbol{B_s} \\
\boldsymbol{C_s} & \boldsymbol{D_s} 
\end{bmatrix}	\begin{bmatrix}
\boldsymbol{T} \\
\boldsymbol{q}
\end{bmatrix} = \begin{bmatrix}
\boldsymbol{S_T} \\
\boldsymbol{S_q}
\end{bmatrix},		
\end{equation}

where 
\begin{align*}
    &\boldsymbol{A_s} \in \mathbb{R}^{N_t +1 \times N_t+1}, \quad &\boldsymbol{B_s} \in \mathbb{R}^{N_t+1\times N-1}, \\
    &\boldsymbol{C_s} \in \mathbb{R}^{N-1\times N_t+1}, \quad &\boldsymbol{D_s} \in \mathbb{R}^{N-1\times N-1}, \\
    &\boldsymbol{S_T} \in \mathbb{R}^{N_t+1\times 1}, \quad  &\boldsymbol{S_q} \in \mathbb{R}^{N-1\times 1}, \\
    &\boldsymbol{T} \in \mathbb{R}^{N_t +1 \times 1}, \quad &\boldsymbol{T}_i = T(t_{i-1}), \quad i=1,..,N_t+1, \\
        &\boldsymbol{q} \in \mathbb{R}^{N-1 \times 1}, \quad &\boldsymbol{q}_i = q_i, \quad i=1,..,N-1.
\end{align*}
In the rows of the upper block of the matrix system we enforce the time step scheme based on if a row corresponds to an irregular or regular grid point as discussed previously. In the lower block rows, we enforce the relationship between the state and state augmented variable for the zero order jump condition.

In a similar manner given $\boldsymbol{T}, \boldsymbol{q}, $ and control $v$, augmented IIM for the adjoint equation can be written as the block matrix system:

\begin{equation}
    \begin{bmatrix}
\boldsymbol{A_{\lambda}} & \boldsymbol{B_{\lambda}} \\
\boldsymbol{C_{\lambda}} & \boldsymbol{D_{\lambda}} 
\end{bmatrix}	\begin{bmatrix}
\boldsymbol{\lambda} \\
\boldsymbol{q^{\lambda}}
\end{bmatrix} = \begin{bmatrix}
\boldsymbol{S_\lambda} \\
\boldsymbol{S_{q^{\lambda}}}
\end{bmatrix},		
\end{equation}

where
\begin{align*}
    &\boldsymbol{A_{\lambda}} \in \mathbb{R}^{N_t +1 \times N_t+1}, &\boldsymbol{B_{\lambda}} \in \mathbb{R}^{N_t+1\times N-1}, \\
    &\boldsymbol{C_{\lambda}} \in \mathbb{R}^{N-1,N_t+1}, \quad   &\boldsymbol{D_{\lambda}} \in \mathbb{R}^{N-1\times N-1}, \\
    &\boldsymbol{S_{\lambda}} \in \mathbb{R}^{N_t+1\times1}, \quad &\boldsymbol{S_{q^\lambda}} \in \mathbb{R}^{N-1 \times 1}, \\
    &\boldsymbol{\lambda} \in \mathbb{R}^{N_t +1 \times 1} \quad &\boldsymbol{\lambda}_i = \lambda(t_{N_{t}-i+1}), \quad i=1,..,N_t+1, \\
        &\boldsymbol{q^\lambda} \in \mathbb{R}^{N-1 \times 1} \quad &\boldsymbol{q^\lambda}_i = q_i^\lambda, \quad i=1,..,N-1.
\end{align*}
Like in the case of the state equation, we enforce the scheme at regular and irregular grid points in the upper block, and in the lower block enforce the augmented variable on the state.

A consideration of this paper is not to derive appropriate preconditions to apply to the augmented state or augmented adjoint systems. However, constructing such a preconditioner would help speed up the computational time required to solve the systems.

\section{Discretization of The Objective Function and Gradient}
In Section \ref{S:disstateobj}, we wrote the discretization schemes to produce a discretized state solution $\boldsymbol{T}$, and discretized adjoint solution $\boldsymbol{\lambda}$, for a fixed control $v$. In this section we show how we produce an objective function value and gradient for these discretized solutions.

In our optimal control problem the objective function is:
\begin{equation*}
    \frac{1}{2} \int_{\Omega} (T - \hat{T})^2dt.
\end{equation*}
We approximate the objective function with the trapezoid rule:
\begin{equation}
    \mathcal{J} =  \frac{1}{2} \int_{\Omega} (T - \hat{T})^2dt =  \int_{\Omega} \hat{f}(t)dt  = \sum_{i=1}^{N_t} \frac{\Delta t}{2} \bigg( \hat{f}(t_{i-1}) + \hat{f}(t_i) \bigg) + \mathcal{O}(\Delta t^2),
\end{equation}
where 
\begin{equation}
    \hat{f}(t) = \frac{1}{2} (T(t) - \hat{T}(t))^2.
\end{equation}
For simplicity of notation, we define the numerical approximation of the objective function to be
\begin{equation}
    \hat{J}^{\Delta t} = \sum_{i=1}^{N_t} \frac{\Delta t}{2} \bigg( \hat{f}(t_{i-1}) + \hat{f}(t_i) \bigg).
\end{equation}

We approximate the $i-$th component of the gradient, associated with the binary variable $v_i$, using the trapezoid rule:
 \begin{equation}
      \frac{\partial \mathcal{J}}{ \partial v_i} = -\int_{\Omega_i} C \lambda dt =  \int_{\Omega_i} \hat{g}(t) dt =\sum_{k=1}^{\hat{N}_t^i} \frac{\Delta t}{2} \bigg( \hat{g}(\hat{t}_{k-1}) + \hat{g}(\hat{t}_k) \bigg) + \mathcal{O}(\Delta t^2),
\end{equation}
where 
\begin{equation}
    \hat{g}(t) = -C\lambda(t).
\end{equation} 
The notation $\hat{N}_t^i$ corresponds to the number of time steps in the mesh that exist within $\Omega_i$, and $\hat{t}$ corresponds to the specific time steps that exist in $\Omega_i$.
Again, for simplicity of notation, we define the numerical approximation of $i$-th component of the gradient vector to be 
\begin{equation}
    \bigg(g_{\Delta t}\bigg)_i = \sum_{k=1}^{N_t^i} \frac{\Delta t}{2} \bigg( \hat{g}(\hat{t}_{k-1}) + \hat{g}(\hat{t}_k) \bigg)
\end{equation}

\section{A trust-region method for bang-bang control}\label{s:tr}
In \cite{vogt2020mixed}, a trust-region method is introduced to solve binary partial-differential equation constrained optimization problems. We note that this method is readily applicable to our problem, and is outlined in Algorithm \ref{A:Steepest}. The algorithm can be viewed as solving a sequence of differential equations and binary knapsack problems to yield a local minimizing binary control. A strength of this algorithm is that the computational effort is not bottlenecked by the number of binary variables present in the optimal control problem. This is because the binary variables appear in the trust-region subproblem; the trust-region subproblem is a knapsack problem, so it is computationally inexpensive to solve even for a large amount of binary variables. Instead, the computational effort required to execute the trust-region method is focused on solving the state and adjoint equation, and then building the gradient.

\begin{algorithm}[t!]
\SetAlgoVlined
\caption{Steepest-Descent Trust-Region Algorithm.}\label{A:Steepest} 
Given initial trust-region radius $\Delta_{0} = \bar{\Delta} \geq 1$ and initial guess $v^{(0)} \in \{0,1\}^{N}$\\
Select an acceptance step parameter $\bar{\rho}$, and set $k \gets 0$\\
Evaluate the objective function $\hat{J}_{\Delta t}^{(k)} = \hat{J}^{\Delta t}(v^{(k)})$ and the gradient $g_{\Delta t}^{(k)} = \nabla_{v} \hat{J}^{\Delta t}(v^{(k)})$\;

\While{$\Delta_k \geq 1$}{
  Solve the trust-region (knapsack) subproblem for $\widehat{v}$:
  \begin{align*}
    \widehat{v} = \underset{v}{\text{minimize}} \quad  &g_{\delta t}^{(k)^T}\left(  v -v^{(k)} \right) + \hat{J}_{\delta t}^{(k)} &  \\ 
    \quad & \|v-v^{(k)}\|_1 \leq \Delta_k  &   \\
    &  v\in \{0,1 \}^{N} & 
  \end{align*}
  Evaluate the objective $\hat{J}_{\Delta t}(\widehat{v},T(\widehat{v}))$ by solving state equations with $\widehat{v}$\;
  
  Compute the ratio of actual over predicted reduction: $\rho_k = \frac{\hat{J}_{\Delta t}^{(k)} - \hat{J}_{\Delta t}(\widehat{v},\psi(\widehat{v}))}{-\big( g_{\Delta t}^{(k)} \big)^T \big( \widehat{v} - v^k\big)}$\;

  \uIf{$\rho_k > \bar{\rho}$}{
    Accept the step: $v^{(k+1)} = \widehat{v}$, and evaluate the gradient $g_{\Delta t}^{(k+1)} = \hat{J}'(v^{(k+1)})$\\
    \lIf{$\| v^{(k+1)} - v^{(k)} \|_1 = \Delta_k$,}{
    	increase the trust-region radius $\Delta_{k+1} = 2\Delta_{k}$ 
    }	
  }
  \uElseIf{$\rho_k >0$}{
    Accept the step $v^{(k+1)}= \widehat{v}$, and evaluate the gradient $g_{\Delta t}^{(k+1)} = \hat{J}'(v^{(k+1)})$\\
    Keep trust-region radius unchanged $\Delta_{k+1} = \Delta_k$
  }
  \Else{
    Reject the step, set $v^{(k+1)} = v^{(k)}$, and copy the gradient $g_{\Delta t}^{(k+1)} = g_{\Delta t}^{(k)}$\\
    Reduce the trust-region  radius $\Delta_{k+1} = \text{floor}\left(\frac{\Delta_k}{2}\right)$\;
  } 

  Set $k \gets k+1$\; 
} 
\end{algorithm}

\section{Numerical experiments}\label{numerics}

We provide numerical evidence that our approach to solve the bang-bang state equation, using IIM, provides the correct numerical solution. In our simulations, we solve \ref{strongstate} with parameters $f=0, \; C=3,  \; K=1,  \; T_0=70,  \; T_s=50,  \; t_{\text{final}}=10, N=10.$, and choose $v$ in the following manner: 
\begin{equation}
    \begin{cases}
         v_i= 1 \quad \text{$\; i$ is odd} \\
          v_i= 0 \quad \text{$\;i$ is even},
    \end{cases}
\end{equation}
for $i=1,..,N$.
In this instance, we can derive the exact solution to the state equation:
\begin{equation}
    T_{\text{true}}(t)=\begin{cases}
         T_0e^{-Kt} + T_s - T_se^{-Kt}  + \frac{C}{K} -  \frac{C}{K}e^{-Kt} \quad t\in \Omega_i, \text{$i$ is odd} \\
          T_0e^{-Kt} + T_s - T_se^{-Kt}   \quad t\in \Omega_i, \text{$i$ is even}. \\
    \end{cases}
\end{equation}
We numerically solve \ref{strongstate} in several instances containing a different number of time steps i.e $N_t =2^p$ for $p=5..11$ using both Eulers method and our IIM approach. We record the maximal of the absolute error between the numerical and true solution in each simulation, and calculate the convergence order as we refine the mesh for both Euler method and our IIM approach; we show these results in Table \ref{T:statecomp}. We note that we calculate the error in the infinity norm, and the convergence $\gamma^*$ by:
\begin{align*}
    \gamma^* = \frac{\log\bigg(\frac{E^*_{\frac{N_t}{2}}}{E^*_{N_t}}\bigg)}{\log(2)}.
\end{align*}
In the IIM approach \cite{li2001immersed}, we expect the average convergence order to reflect the overall convergence of the method. If we average the convergence order for the IIM approach in \ref{T:statecomp}, we see that the order of the method does indeed reflect a convergence order of 1. In comparison, as we observe the convergence order for Euler's method, we see that the convergence order average is not close to 1. Instead, as we refine the mesh the convergence order is approaching 0. This is not surprising, because the theoretical assumption to use Euler's method is $T\in C^2(\Omega)$, which is violated in this case. In the absence of such an assumption, one cannot guarantee the consistency, let alone the convergence of the numerical method. In Figure \ref{F:solution}, we visually show the solutions generated from the IIM and Euler approach. We observe that the IIM captures the jumps in the state solution due to the bang-bang control, while Euler's method does capture these jumps accurately. It is crucial to solve the state equation correctly, and as we have demonstrated, the introduction of bang-bang controls will cause traditional numerical methods for ordinary-differential equations to fail.

On the optimal control front, being unable to solve the state equation correctly undoubtedly produces an incorrect adjoint solution. It should come as no surprise that the gradient, which is calculated with this incorrect state and adjoint solution, also will be incorrect. In Table \ref{T:fullcomp} we demonstrate what occurs when we use the IIM approach to calculate the state and adjoint equation. We first observe that both the state and the adjoint solutions achieve first-order convergence. This is to say that the state solution is first order accurate, but also when we pass this solution to the adjoint equation and use the IIM approach we maintain the first order accuracy. We use the state and adjoint solution to then create a gradient. In the table, we compare the difference of the gradient and the a forward difference approximation of the gradient in the infinity norm. We observe that the gradient also obtains first order convergence, which is expected. Our results demonstrate that not only using our IIM approach solves the state and adjoint equation in the presence of bang-bang control correctly, but also that the gradient produced by using the IIM generated state and adjoint solution is consistent. 

Since we have shown that our approach creates a consistent gradient, we now share results for our trust-region method applied to the problem. In our runs, we let $T_0=10,\; T_s=0 \; K=0.1 \; C=2 \; f=0, \; g=0 $ and $t_{\text{final}}=100$. In our numerical tests, we let $\hat{T}(t) = 5 + 0.5\sin(t)$. We test the trust-region method when the number of bangs are 100, 1000, and 10000. For all simulations we pick $N_t=20000$ to guarantee a resolved bang-bang control for all the scenarios mentioned. 

To execute our trust-region method, we need a feasible starting guess. While a random binary solution would be sufficient, in \cite{vogt2020mixed} it was observed that solving the relaxation of the problem ($0\leq v \leq 1)$ and then rounding was produced a minimum with lower objective function value. We choose to use the simple rounding heuristic:
\begin{align*}
    [v_{\text{round}}]_i = \begin{cases}
         0 \quad v_i < 1 \\
         1 \quad v_i \geq 1.
    \end{cases}
\end{align*}

In Table \ref{T:CPU} we report the computational effort required to solve the relaxation, the number of trust-region iterations required to find a minimizing solution, and the total time required to do the two pieces of the trust region method: solving the state and adjoint equation then constructing the gradient and executing the trust-region subproblem. We observe the amount of time taken to solve the knapsack problem is computationally negligible when compared to the time required to build a gradient. This means that with our trust-region method for binary optimization the bottleneck is the bottleneck of doing classical optimal control; we observe that the binary variables are almost inconsequential. Solving the adjoint and state equation in the most efficient manner possible is out of the scope of this work, but we note that carefully constructed preconditioners for the state and adjoint equation can also drastically decrease the computational time to solve the state and adjoint systems. 

In Figure $\ref{fig:tr}$, we show the progress of the trust-region methods radius and objective-function value as a function of iteration count. We observe that as we increase the amount of bangs present, we increase the amount of trust-region iterations required to find a minimizing solution. Moreover we reduce the objective-function value quite quickly early in execution of the trust-region method, but then see slow convergence to a local minimum; an artifact of the non-convexity of the objective function.
 
In Table \ref{T:obj} we report the objective function value and in Figure \ref{fig:stateandcontrol} we show the optimal control, state, and target function for a different amount of bang controls. In the table, we see that we are able to make a significant percent reduction in all cases, and that the percent reduction is observed to be greater when more bangs are present. In the figure, we observe that we are successful at producing controls that create states which model the behavior of the target we are attempting to match. We observe that as we increase the number of bangs, we find an optimal control that makes the state model the target function more accurately. These results indicate that using our IIM approach, combined with an adjoint based gradient and trust-region method, does indeed yield a successful control strategy.

\begin{table}[htbp]
\centering
\caption{Error (E) and convergence ($\gamma$)  table for Euler (superscript E) method and the IIM (superscript IIM) approach applied to the state equation \ref{strongstate}. }\label{T:statecomp}
\begin{tabular}{|c||r|r|r|r|}
\hline
\multicolumn{1}{|c||}{} & \multicolumn{2}{c|}{IIM} &  \multicolumn{2}{c|}{Euler Method} \\
$N_t$  & $E_{N_t}^{\text{IIM}}$ & $\gamma_{N_t}^{\text{IIM}}$& $E_{N_t}^{\text{E}}$ & $\gamma_{N_t}^{\text{E}}$ \\ \hline \hline
32 &    3.140         &           &   3.140 &                  \\ \hline
64 &         1.255    &          1.322           &  1.517 &  1.048                     \\ \hline
128 &           0.574   &             1.124        & 1.431 &    0.084                 \\ \hline
256 &       1.497 &                 -1.382   & 1.369 &              0.064   \\ \hline
512 &       0.135  &                    3.467  &  1.329 &           0.042      \\ \hline
1024 &    0.067&                    1.013   &   1.328    &             0.001      \\ \hline
2048 &  0.034    &                  0.944       &    1.323   &              0.005      \\ \hline
\end{tabular}
\end{table}

\begin{figure}[htbp]
  \centering
\includegraphics[width=\textwidth]{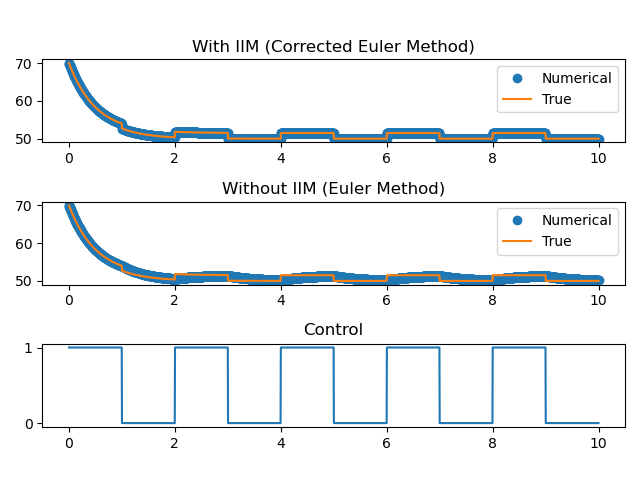}
\caption{Comparison of the numerical solution yielded by the IIM and Euler method to the exact solution for a fixed bang-bang control for $N_t=2048$ from Table \label{F:solution}}
\end{figure}

\begin{table}[htbp]
\centering
\caption{Error and convergence table for the state, adjoint and gradient .}\label{T:fullcomp}
\begin{tabular}{|c||r|r|r|r|r|r|}
\hline
\multicolumn{1}{|c||}{} & \multicolumn{2}{c|}{State} &  \multicolumn{2}{c|}{Adjoint} &   \multicolumn{2}{c|}{Gradient} \\
$N_t$  & $E_{N_t}^{\text{IIM}}$ & $\gamma_{N_t}^{\text{IIM}}$& $E_{N_t}^{\text{IIM}}$ & $\gamma_{N_t}^{\text{IIM}}$ & $E_{N_t}^{\text{IIM}}$ & $\gamma_{N_t}^{\text{IIM}}$ \\ \hline \hline
32 &    19.834         &           &   44.528 &  &36.177 &                \\ \hline
64 &         7.607    &   1.382               &  9.945 &  2.162           & 19.417&       0.897    \\ \hline
128 &           3.375   &      1.172               & 4.283 &    1.215  & 9.933&             0.966   \\ \hline
256 &       1.678 &         1.008   & 1.369 &              1.126  &4.919 & 1.013 \\ \hline
512 &       0.832  &            1.011  &  1.329 &           1.090   & 2.517&    0.966\\ \hline
1024 &    0.415&                    1.005   &  0.448   &             1.039  &1.263 & 0.994    \\ \hline
2048 &  0.207    &                  1.006       &   0.223   &              1.007    &0.632 & 0.998   \\ \hline
\end{tabular}
\end{table}

\begin{figure}[htbp]
\centering
\subfloat[100 bangs.]{\label{fig:tr1}\includegraphics[width=0.45\textwidth]{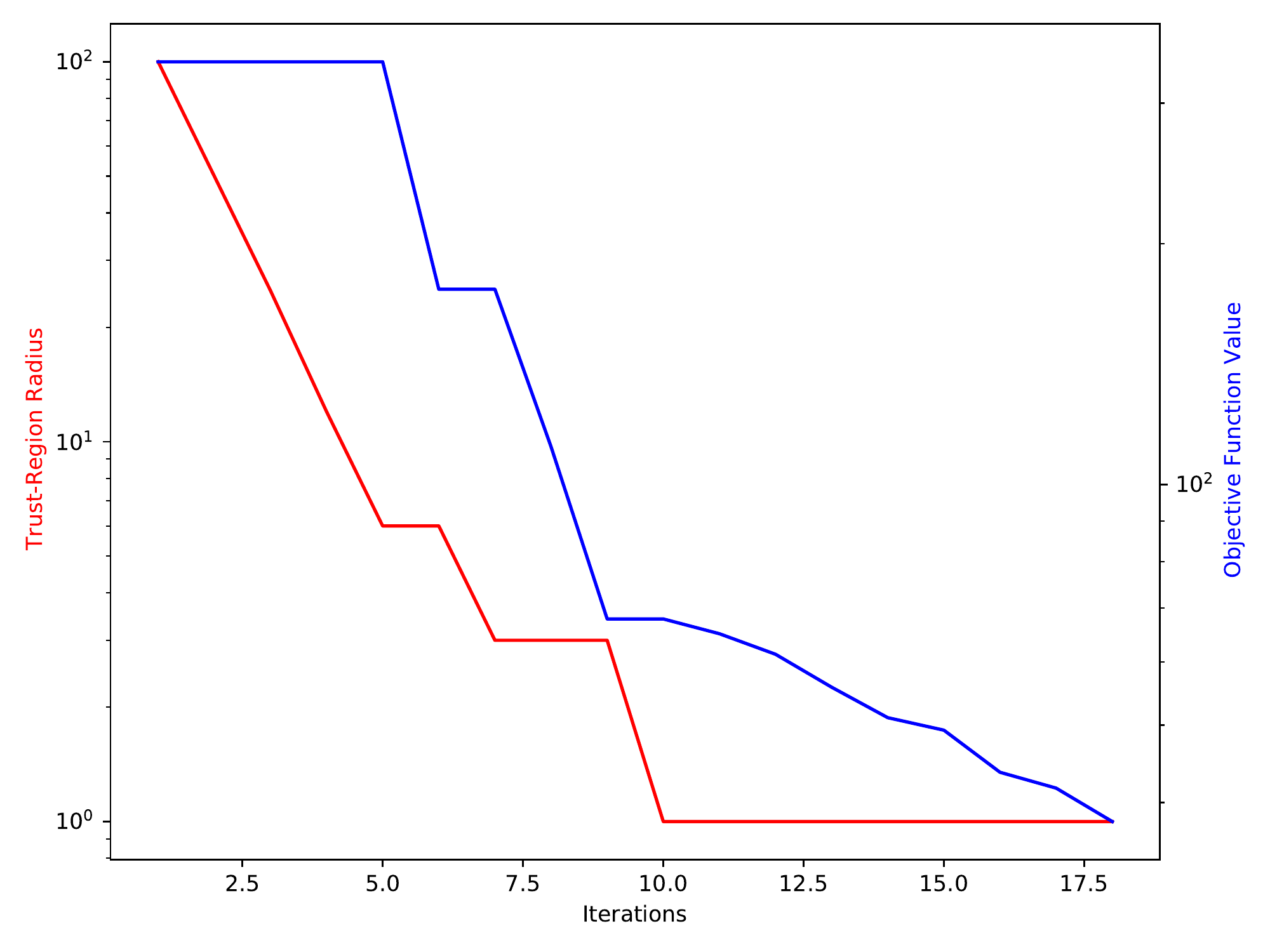}}
\subfloat[1000 bangs.]{\label{fig:tr2}\includegraphics[width=0.45\textwidth]{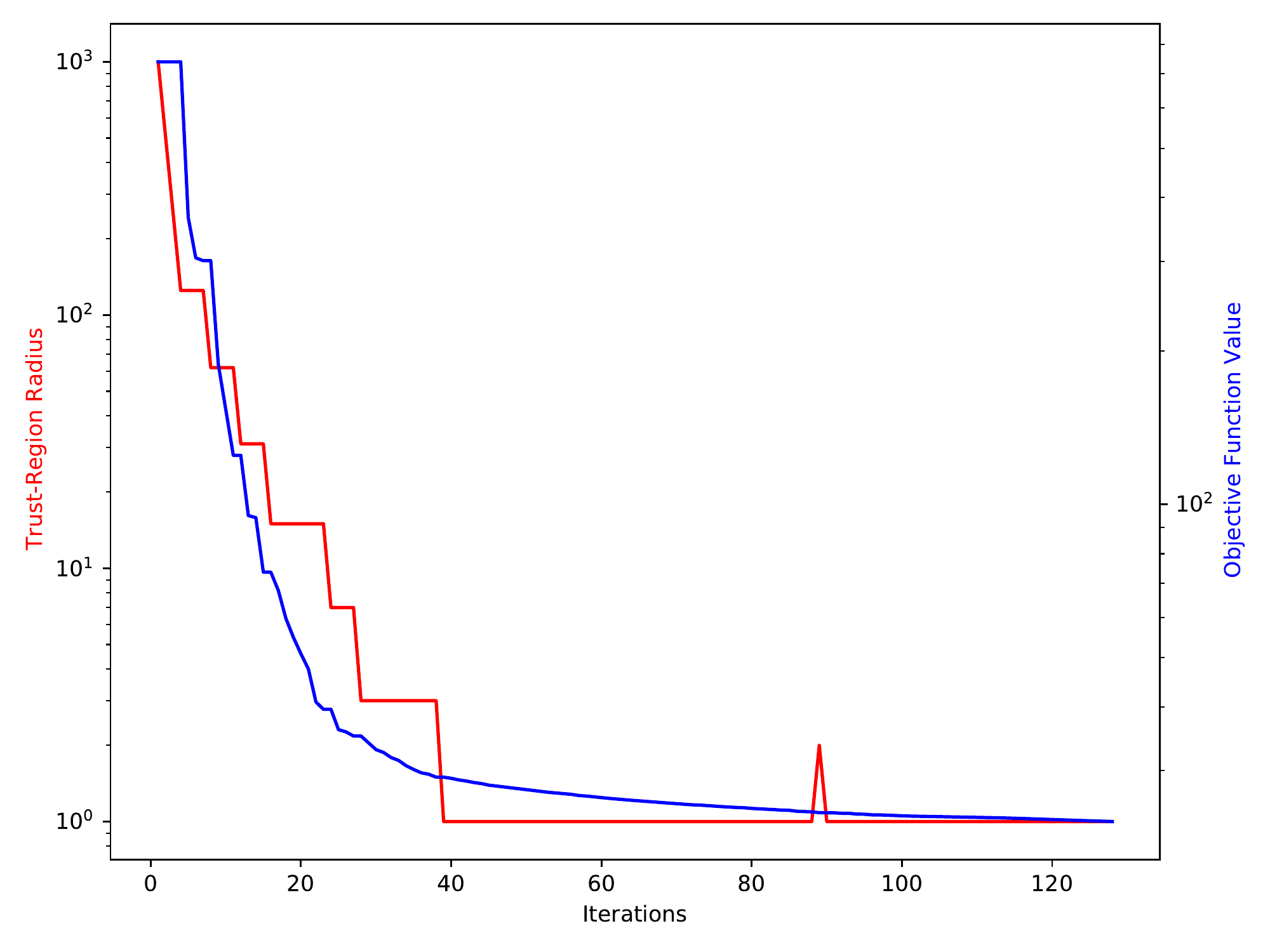}}
\begin{center}
\subfloat[10000 bangs.]{\label{fig:tr3}\includegraphics[width=0.45\textwidth]{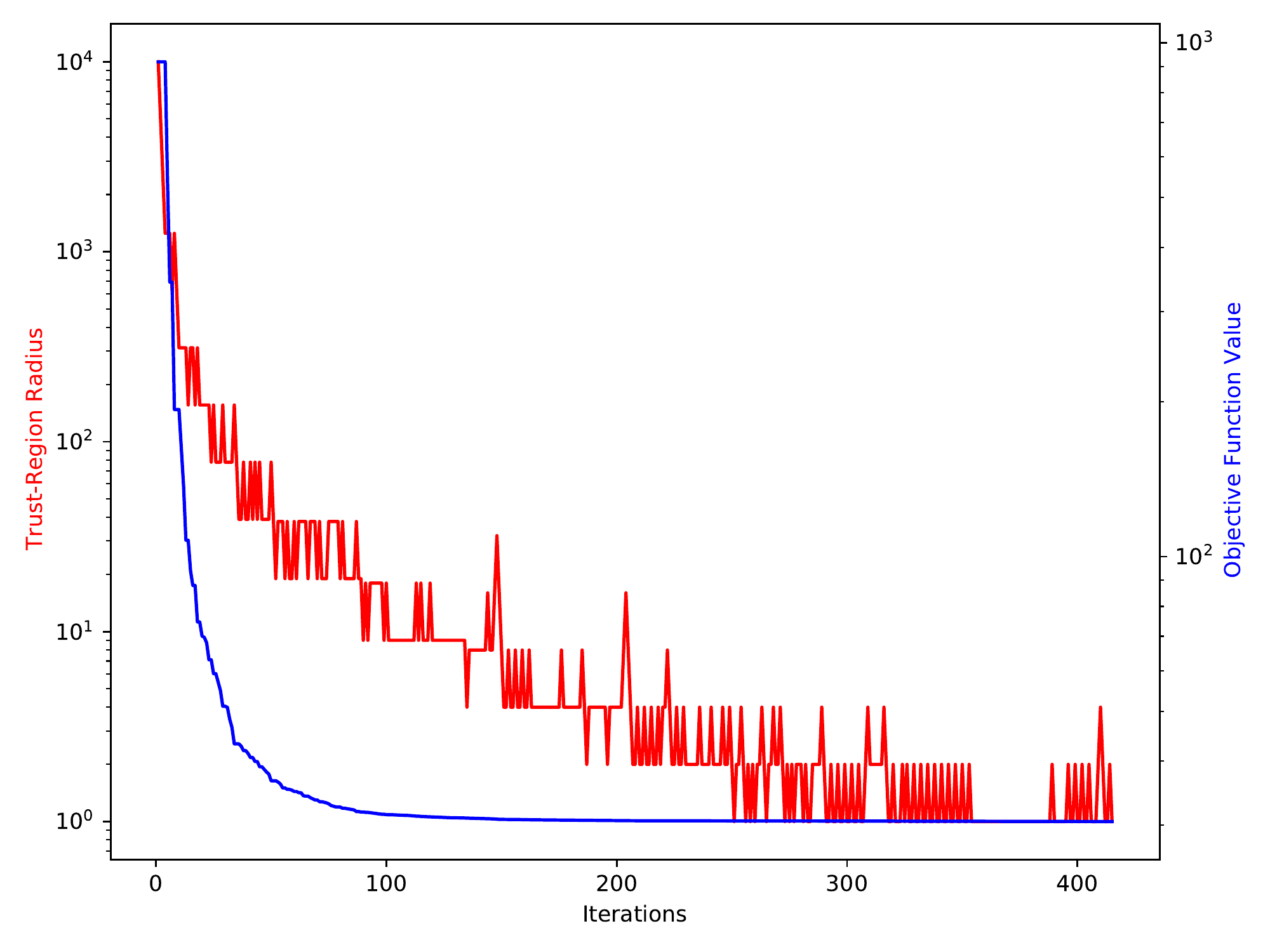}}
\end{center}
\caption{Trust-region method objective function value and trust-region radius evolution for 100,1000, and 10000 binary variables. On the left axis (red) we plot the objective function value as a function of iteration count, and on the right axis (blue) we plot the trust-region radius as a function of iteration count.}
\label{fig:tr}
\end{figure}

\begin{figure}[htbp]
\centering
\subfloat[State for 100 BV.]{\label{fig:rec1}\includegraphics[width=0.32\textwidth, trim=20 0 20 20, clip]{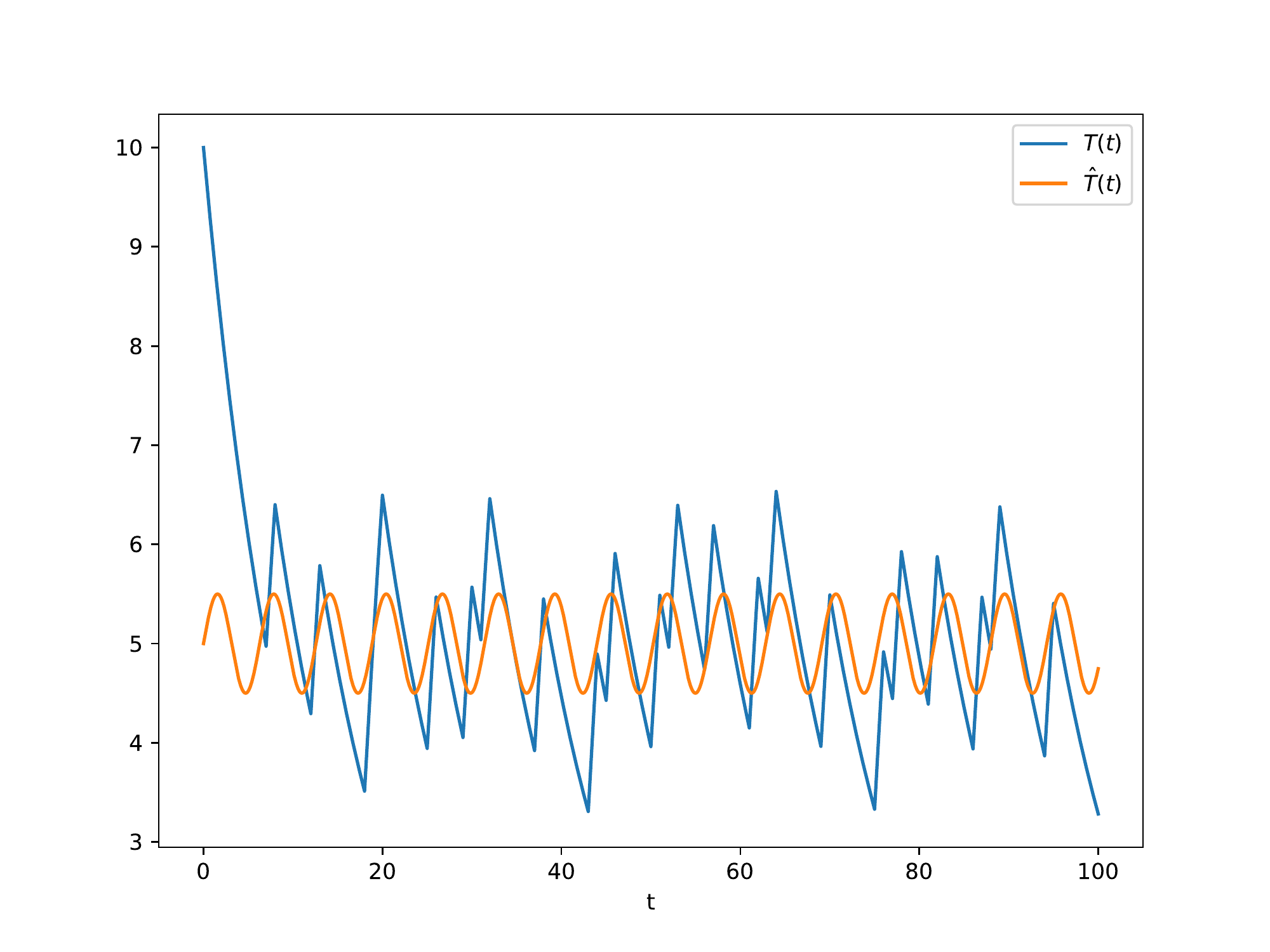}}
\subfloat[State for 1000 BV.]{\label{fig:sq1}\includegraphics[width=0.32\textwidth, trim=20 0 20 20, clip]{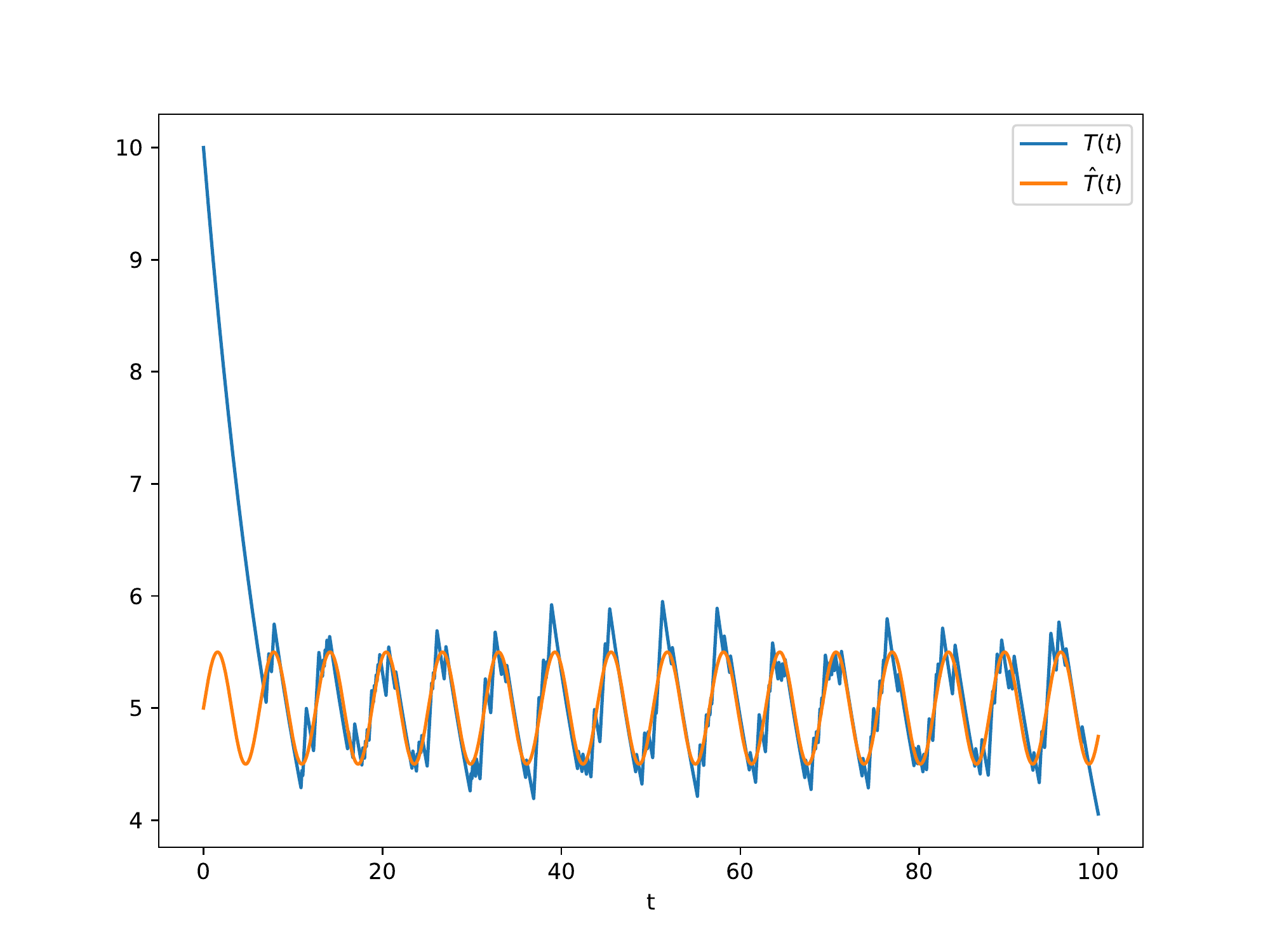}}
\subfloat[State for 10000 BV.]{\label{fig:circ1}\includegraphics[width=0.32\textwidth, trim=20 0 20 20, clip]{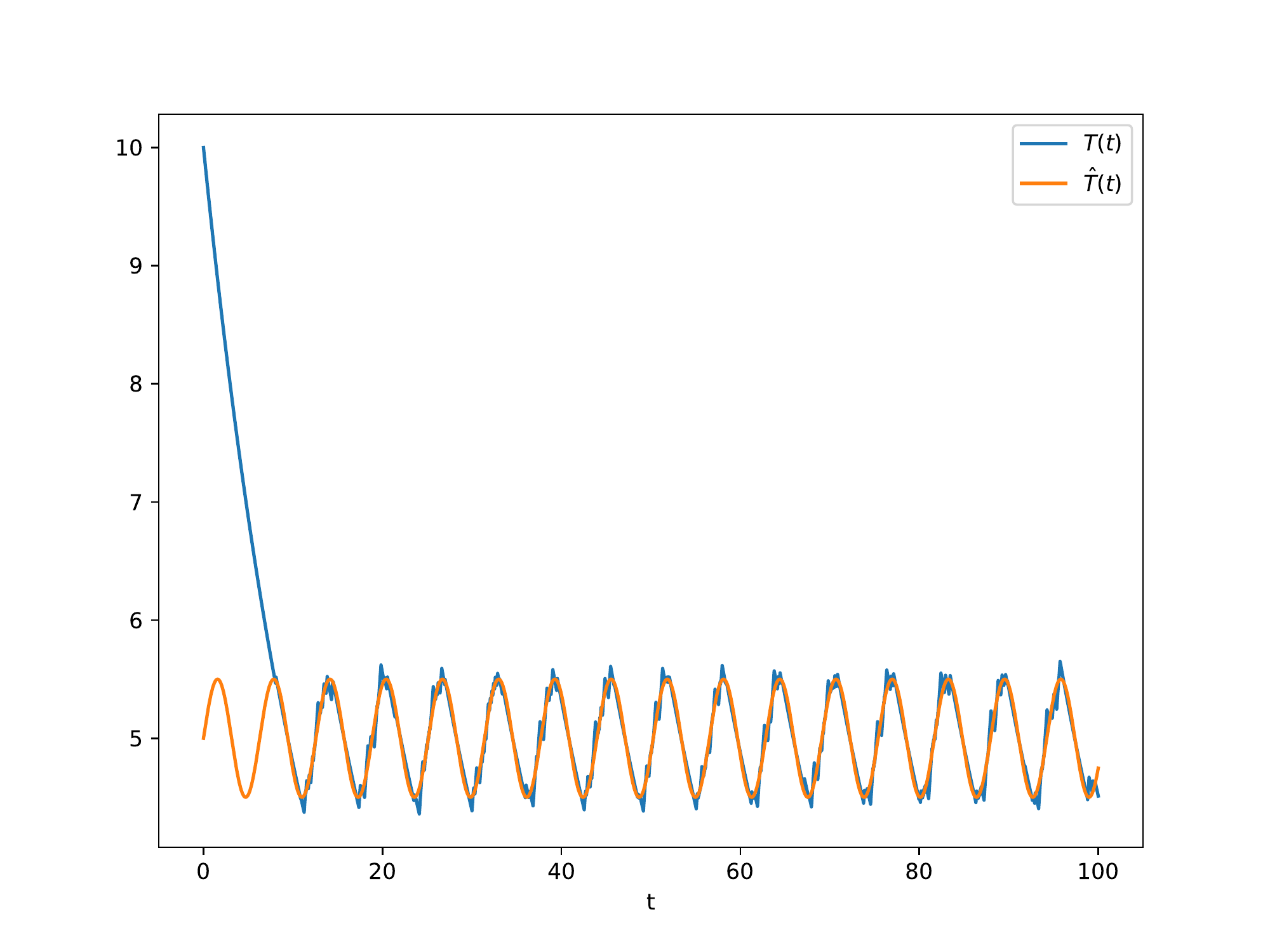}}\\
\subfloat[Control for 100 BV.]{\label{fig:rec2}\includegraphics[width=0.32\textwidth, trim=20 0 20 20, clip]{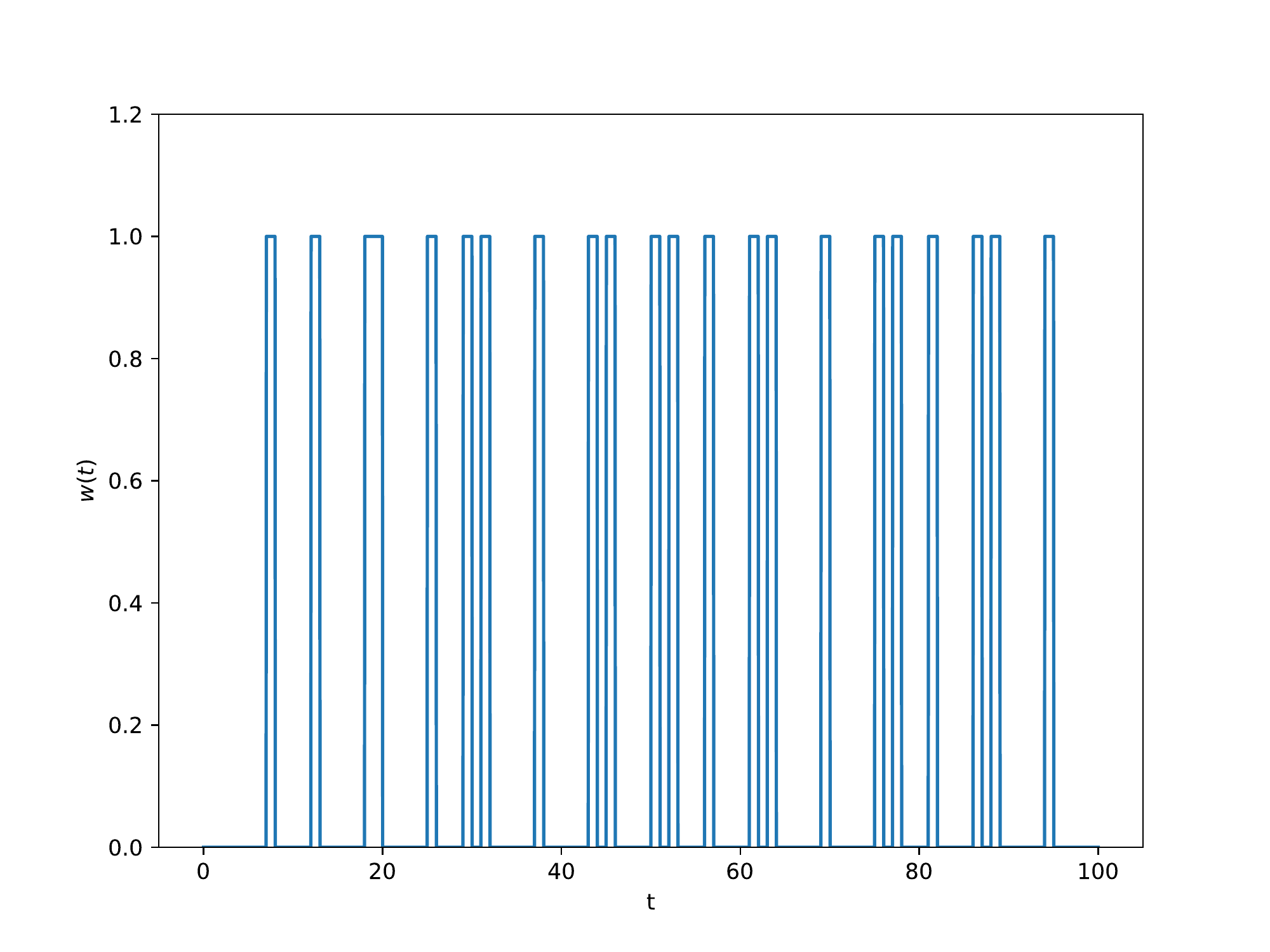}}
\subfloat[Control for 1000 BV.]{\label{fig:sq2}\includegraphics[width=0.32\textwidth, trim=20 0 20 20, clip]{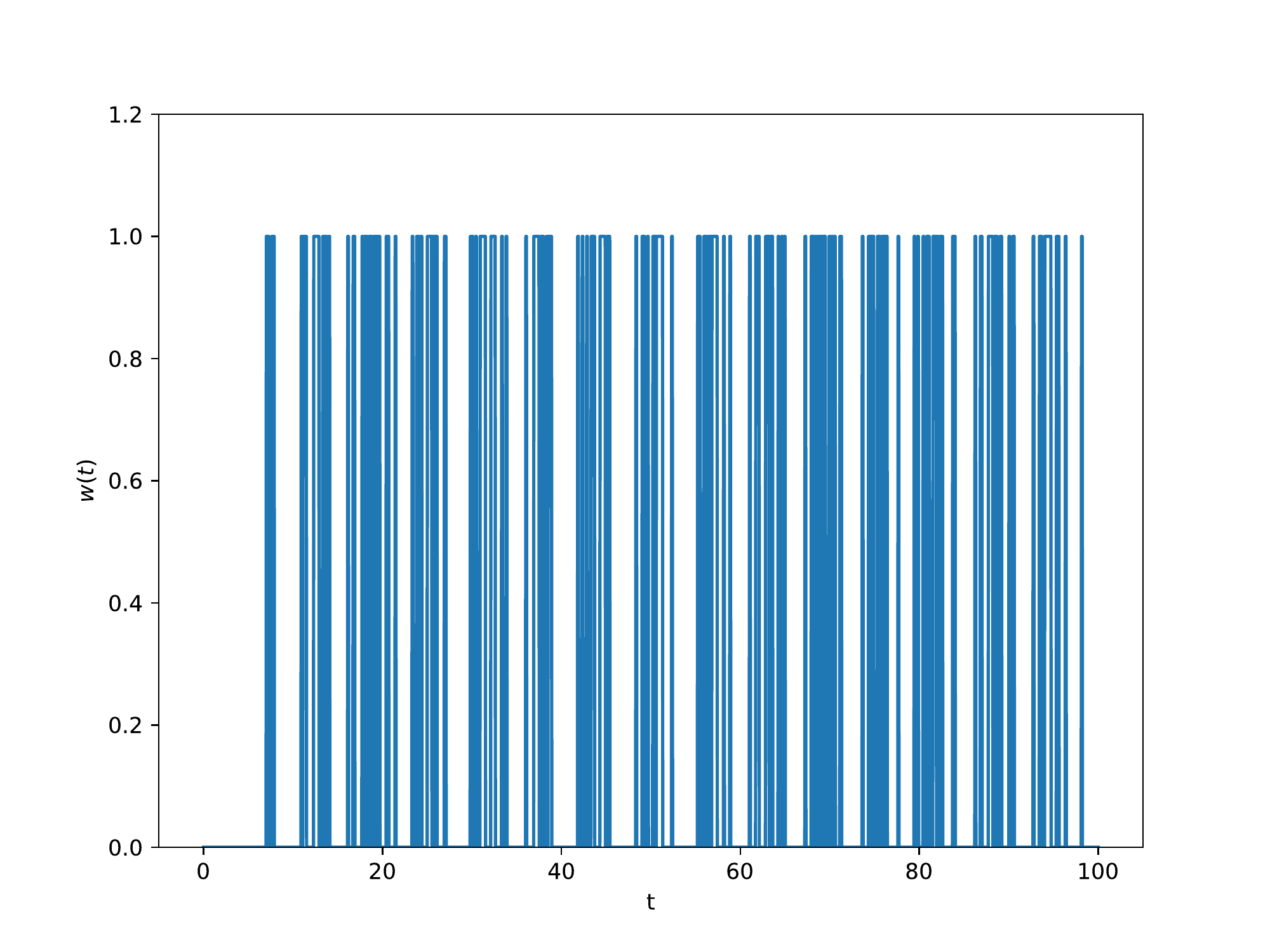}}
\subfloat[Control for 10000 BV.]{\label{fig:circ2}\includegraphics[width=0.32\textwidth, trim=20 0 20 20, clip]{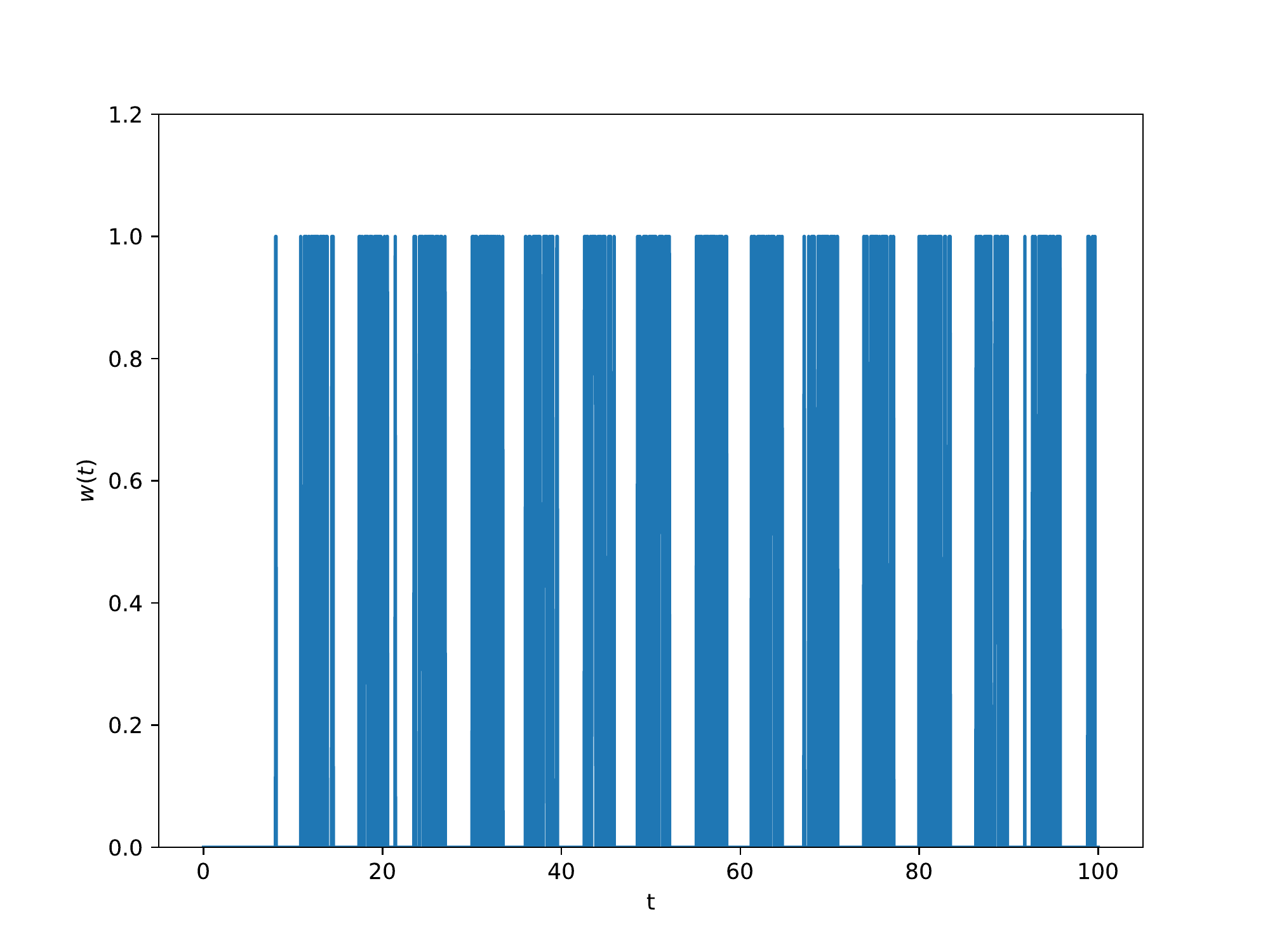}}
\caption{Plots of the state $T(t)$, target $\hat{T}(t)=5+0.5\sin(t)$, and bang control $w(t)$ for 100, 1000, and 10000 binary variables (BV). }
\label{fig:stateandcontrol}
\end{figure}

\begin{table}[htbp]
\centering
\caption{For 100, 1000, and 10000 binary variables we record the CPU times(s) for solving the relaxation to initial the trust-region method. In addition we record the CPU times(s) for computing the state, adjoint, and gradient; in addition to the time required to executing the trust-region sub problems (knapsack problems). These times are reflective of the amount of trust-region iterations required to find a minimizing bang-bang control.}\label{T:CPU}
\begin{tabular}{|r||r|r|c|c|}
\hline
\multicolumn{1}{|c||}{} & \multicolumn{3}{c|}{Solution CPU Time(s)} & \multicolumn{1}{c|}{Trust-Region} \\
 Number of Bangs &  Relaxation & ODEs/Gradient & MIPs & Iterations \\ \hline \hline
100 & 7999.89            &9372.31                 &         1.31                        &    18            \\ \hline
1000 & 5270.65            & 72173.45                   &            8.70                                                  &    128                  \\ \hline
10000 & 5637.85              & 720508.67                     &           218.09                                                &    415         \\ \hline
\end{tabular}
\end{table}

\begin{table}[htbp]
\centering
\caption{The objective function value before (OFVB) at the rounded relaxation (start of trust-region), objective function value after (OFVA) at the termination of trust-region, and the percent-reduction in the objective function value from the start to end of the trust-region method for 100, 1000 and 10000 binary variables.}\label{T:obj}
\begin{tabular}{|r||l|l|c|}
\hline
 Number of Bangs &  OFVB & OFVA & Percent Reduction ($\%$) \\ \hline \hline
100 & 337.93             &37.93                 &         88.77                      \\ \hline
1000 & 739.24           &   23.84                 &           96.77                                                              \\ \hline
10000 & 918.76             & 30.50                     &           96.68                                                          \\ \hline
\end{tabular}
\end{table}

\section{Solving general binary bang-bang problems with our approach}
In this Section we summarize how the our approach can be applied to more general reduced space binary bang-bang problems. We discuss the three key components in as much generality as possible : building the Lagrangian, building the IIM schemes for the state and adjoint system, and applying the trust-region method.\\
\textbf{Construct a Lagrangian.}\\
Suppose we have a bang-bang optimal control problem of the form:

\begin{equation}\label{eq:moregenproblem}
\begin{aligned}
& \underset{T,\; v}{\text{minimize}} & & \mathcal{J}(T(v),v)\\
& \text{subject to} 
  & &  C(T,v)=0 \quad \text{ on} \; \Omega,\\
   & & &   \Omega = \bigcup_{i=1}^{N }\Omega_i, \\
   & & &   w(t) = \sum_{i=1}^{N} v_i \mathbb{1}_{\Omega_i \setminus \partial \Omega_i  }(t)  \\,
    & & &  v_i \in \{0,1\} \quad \forall i=1...N,   \\
    & & & \Omega_i =[\tau_i,\tau_{i+1}]  \quad \forall \; i=1,...,N, \\
    & & & \tau_1 = 0, \quad \tau_{N+1} = t_{\text{final}}, \\
   & & & \Omega = [0,t_{\text{final}}],   \\
   & & & T(0) = T_0.
    \end{aligned}
\end{equation}

Where $C:t \rightarrow \mathbb{R}^k, $ $k\geq 1$. We define an adjoint variable, $\lambda:t \rightarrow \mathbb{R}^k$, and define the Lagrangian
\begin{equation}
    \mathcal{L}(T,\lambda,v) = \mathcal{J} + \int_{\Omega} \lambda^T Cdt.
\end{equation}
From the Lagrangian, we derive the restricted state and adjoint equation, and the gradient, as seen in Section \ref{s:laggrad}. We note that if there are additional constraints on the binary variable, which is not considered in our formulation, these constraints can be ignored in the Lagrangian and instead enforced in the trust-region subproblem. 

\textbf{Derive the augmented IIM Schemes for the state and adjoint system.}
Next, the state and adjoint problem are discretized using the IIM approach. Typically, we use a standard numerical method (in this paper Euler), as a foundation to construct a "corrected" scheme with the IIM approach. One can view our approach as the standard scheme being applied at regular grid points, while the "corrected" scheme is applied at irregular grid points. If the scheme that is used at the regular grid points has been shown to be consistent, stable, and convergent, then the scheme at the irregular grid points consistency, stability, and convergence proofs need to be only slightly altered. Having a methods that are consistent, stable, and convergent for the state and adjoint equation are critical. Because without the correct state and adjoint computation, it is impossible to have a consistent gradient, let alone to solve the overarching optimal control problem with any numerical optimization approach.   
\\
\textbf{Apply the trust-region method.}
In our work we first solve the relaxation of our optimal control problem, and then round the relaxation to produce a feasible starting guess to the binary optimal control problem. In our previous work, \cite{vogt2020mixed} we have observed that initializing our trust-region method with the rounded relaxation produced controls that yield lower objective function value than we when initialized our initial trust-region method with a random binary guess. Our trust-region methods computational effort is bottle-necked by only the state and adjoint computation, which one would have to perform for any reduced space optimal control problem. However, regardless of the amount of binary variables present in the problem we can solve the trust-region sub problem quickly.

\section{Conclusions}
In this work, we introduced a formal Lagrangian approach for solving bang-bang optimal control problems. We introduced a target optimal control and wrote the optimality conditions, which led to solving the state and adjoint equation using IMM to ensure that we obtain a numerical solutions that are consistent, stable, and convergent. By solving the state and adjoint equation we are able to produce a gradient, and an objective function value that is associated with a fixed bang-bang control. We use the objective function value and gradient to execute our trust-region method. We showed numerical simulations that both demonstrated the importance of not applying standard numerical methods to the state equation, but also that our approach generates bang-bang controls that meets our goals. This work provides a blueprint that readers can use to as a foundation to solve their own bang-bang problems quickly and efficiently.

\section{Acknowledgements}
This document was prepared as an account of work sponsored by an agency of the
United States government. Neither the United States government nor Lawrence
Livermore National Security, LLC, nor any of their employees makes any warranty,
expressed or implied, or assumes any legal liability or responsibility for the accuracy,
completeness, or usefulness of any information, apparatus, product, or process disclosed,
or represents that its use would not infringe privately owned rights. Reference herein to
any specific commercial product, process, or service by trade name, trademark,
manufacturer, or otherwise does not necessarily constitute or imply its endorsement,
recommendation, or favoring by the United States government or Lawrence Livermore
National Security, LLC. The views and opinions of authors expressed herein do not
necessarily state or reflect those of the United States government or Lawrence Livermore
National Security, LLC, and shall not be used for advertising or product endorsement
purposes. This document has been approved to be released for widespread consumption according to order LLNL-JRNL-821389. In addition, this material was based upon work supported by the U.S. Department of Energy, Office of Science, 
Office of Advanced Scientific Computing Research, Scientific Discovery through Advanced Computing (SciDAC) program through the FASTMath Institute under Contract DE-AC02-06CH11357 at Argonne 
National Laboratory.  This work was also supported by the U.S. Department of Energy 
through grant DE-FG02-05ER25694.

\bibliography{references}